\documentclass[12pt,letterpaper]{amsart} % LatexMasterTemplate May 2020 on overleaf
\usepackage{palatino, euler, epic,eepic, amssymb, xypic, floatflt, microtype}
\usepackage{verbatim,color}

\usepackage{tikz}
\usetikzlibrary{decorations.markings}
\definecolor{Burgundy}{RGB}{144,0,32}
\usepackage{float}
\input xy
\xyoption{all}

\newtheorem{tm}{Theorem}[section]
\newtheorem{pr}[tm]{Proposition}
\newtheorem{lm}[tm]{Lemma}

\newtheorem{df}[tm]{Definition}
\newtheorem{rem}[tm]{Remark}

\newcommand{\R}{{\mathbb R}}

\newcommand{\Q}{{\mathbb Q}}
\newcommand{\C}{{\mathbb C}}

\newcommand{\Z}{{\mathbb Z}}
\newcommand{\z}{{\zeta}}

\usepackage{colonequals}
\usetikzlibrary{calc}
\usetikzlibrary{decorations.pathreplacing}
\usetikzlibrary{arrows,chains,matrix,positioning,scopes}

\makeatletter
\newcommand{\oset}[3][0ex]{%
  \mathrel{\mathop{#3}\limits^{
    \vbox to#1{\kern-2\ex@
    \hbox{$\scriptstyle#2$}\vss}}}}
\makeatother

\begin{document}

\pagestyle{plain}
\title{Distribution of values of higher derivatives of $L'(s,\chi)/L(s,\chi)$}
\author{Samprit Ghosh}
\address{Mathematical Sciences 468, University of Calgary, \newline 2500 University Drive NW, Calgary, Alberta, T2N 1N4, Canada. }
\email{samprit.ghosh@ucalgary.ca}
\thanks{}
\subjclass[2020]{11M06, 11M41, 11R42}
\keywords{Dirichlet L-functions, density function, log derivatives, Li coefficients}
%\date{May 8, 2024.}
\begin{abstract}
	In this article, we study the value distribution theory for the first derivative of the logarithmic derivative of Dirichlet $L$-functions, generalizing certain results of Ihara, Matsumoto et al. related to ``$M$-functions'' for $\sigma = \operatorname{Re}(s) > 1$. We then discuss the main obstruction toward generalization to higher derivatives.
\end{abstract}
\maketitle
\tableofcontents
\section{Introduction}
Let $K$ be either $\Q$ or an imaginary quadratic number field. In particular, $K$ has exactly one archimedean prime, denoted by $\wp_{\infty}$ (say). Let $\chi$ run over all Dirichlet characters on $K$ whose conductor (the non-archimedean part) is a prime divisor and that satisfy $ \chi(\wp_{\infty}) = 1$.

The average of a complex-valued function $\phi(\chi)$ over a family of $\chi$ as defined above is taken as follows:
$$\text{Avg}_\chi \phi(\chi) = \lim_{m \rightarrow \infty } \; \text{Avg}_{N(\textbf{f}) \leq m} \phi(\chi),$$
where $$ \text{Avg}_{N(\textbf{f}) \leq m} \phi(\chi) = \dfrac{\sum_{N(\textbf{f}) \leq m} \; \left( \sum_{\textbf{f}_\chi = \textbf{f}} \; \phi(\chi)  \right) \; / \; \sum_{\textbf{f}_\chi = \textbf{f}} \; 1 }{\sum_{N(\textbf{f}) \leq m} \; 1}.$$
For the above setting, the following distribution theorem was proved by Ihara in \cite{iharaM}.
\begin{tm}\label{ihdist} \textbf{(Ihara)}
	For $K$ as above and for $\sigma = \operatorname{Re}(s) >1$, there exists a real-valued $C^{\infty}$ function $M_{\sigma} : \C \rightarrow \mathbb{R}$ satisfying $M_{\sigma} (w) \geq 0$ and $ \; \int_{\C} M_{\sigma}(w) \; |dw| = 1$, such that
	\begin{equation}
		\text{Avg}_\chi \; \Phi \left(\frac{L'(\chi, s)}{L(\chi, s)} \right) \; = \; \int_\C M_{\sigma}(w) \; \Phi(w) \; |dw|,
	\end{equation}
	holds for any continuous function $\Phi$ on $\C$. Moreover,
	$$	\text{Avg}_{\chi} \; \psi_z \left(\frac{L'(\chi, s)}{L(\chi, s)} \right) \; = \tilde{M}_{\sigma}(z),$$
	where $\tilde{M}_{\sigma}(z)$ is obtained from the Fourier transform of $M_{\sigma}(z)$ in the sense that 
	$$ \tilde{M}_{\sigma}(z) = \int_\C M_{\sigma}(w) \; \psi_z(w) \; |dw|.$$ 
	Here $\psi_z : \C \rightarrow \C^1$ is the additive character $\psi_z (w) = \exp(i \cdot \operatorname{Re}(\overline{z} w))$.
\end{tm}

\begin{rem}
	Note that Ihara proves this more generally. He considers certain function fields of one variable over a finite field (the theorem is true in this case for $\sigma >3/4$). He also treats the case $K = \Q$ where $\chi$ runs over characters of the form $N(\wp)^{- \tau i}$, and the case where $K$ is a number field having more than one archimedean prime and $\chi$ runs over all ``normalized unramified Gr\"ossencharacters'' of $K$, modifying the definition of average accordingly.
\end{rem}
\begin{rem}
	In a later paper \cite{iharamatsu}, Ihara and Matsumoto proved the above theorem for $\sigma \geq \frac{1}{2}+\epsilon$, under GRH and with ``mild'' conditions on the test function; namely, either $\Phi(w) \ll e^{a|w|}$ holds for some $a>0$, or $\Phi$ is the characteristic function of either a compact subset or the complement of such a subset.
\end{rem}

A similar distribution result for real characters was also proved by Murty and Mourtada; see \cite{mariam1}. For a fundamental discriminant $D$ and the associated real character $\chi_D$, define
$$N(y) := \{ |D| \leq y \; : \; D \text{ is a fundamental discriminant} \}.$$
Then the following theorem holds.
\begin{tm} \textbf{(Mourtada, Murty)}
	Let $\sigma > \frac{1}{2}$ and assume GRH. Then there exists a density function $\mathcal{Q}_{\sigma}(x)$ such that, for any bounded continuous function $\Phi$ on $\R$,
	$$\lim_{y \rightarrow \infty} \frac{1}{N(y)}\sum \limits_{\substack{ |D| \leq y \\ D \text{ fund. disc.}}} \Phi \left( \frac{L'(\sigma, \chi_D)}{L(\sigma, \chi_D)}\right) = \frac{1}{\sqrt{2 \pi}} \int_{ - \infty}^{\infty} \mathcal{Q}_{\sigma}(x) \Phi(x) dx.$$
	It also holds when $\Phi$ is the characteristic function of either a compact subset of $\R$ or the complement of such a subset.
\end{tm}

Our goal is to prove similar distribution theorems for higher derivatives of the logarithmic derivative of Dirichlet $L$-functions. Such distribution results might be of independent interest. However, our key motivation for this work comes from Li's coefficients. For a number field $K$, Li in \cite{licoeff} introduced the sequence of numbers
\begin{equation}
	\lambda_n (K) = \frac{1}{(n-1)!} \frac{d^n}{ds^n} \left(s^{n-1} \log \xi_K(s) \right) \Big|_{s=1} \text{ \hspace{1cm} for } n \geq 1, 
\end{equation}
where $\xi_K(s)$ is the completed Dedekind zeta function of $K$. He proved the following result.

\textbf{Li's Criterion.} 
The generalized Riemann hypothesis for $\z_K(s)$ holds if and only if $\; \lambda_n (K) \geq 0$ for all $n \geq 1$. Later, Li's original criterion was extended to automorphic $L$-functions by Lagarias in \cite{lagariasli}. Brown in \cite{brownli} attempted to prove an effective version of Li's Criterion, showing that non-negativity of the first few $\lambda_i$'s gives zero-free regions of a certain shape around $s=1$. In particular, the single condition $\lambda_2(K) \geq 0$ should imply nonexistence of the exceptional \emph{Siegel zeros} (see \cite[Theorem 5]{brownli}). However, subsequent authors have noted certain errors in Brown's paper.

\textbf{Question.} The signs of individual Li coefficients are subtle and hard to determine. What about their average behavior? More precisely, for a fixed modulus $q$, what can be said about the Li coefficients averaged over Dirichlet characters $\chi$ modulo $q$?

The distribution results for higher logarithmic derivatives proved in this paper can be viewed as a step toward answering this question. In a recently submitted preprint \cite{ghosh2}, the author also studied, for Dirichlet $L$-functions, the $(a,b)$-th moments of the higher derivatives of $L'/L(s,\chi)$ at $s=1$, toward the same goal. The study of higher logarithmic derivatives at $s=1$ is also closely related to higher Euler--Kronecker constants of number fields; see \cite{ghosh1} for related results in that direction.
For the rest of the paper, we write $$\mathcal{L}(s,\chi) := L'(s,\chi) / L(s,\chi).$$
The main result of this article is stated below.
\begin{tm}
	For any $s \in \C$ with $\sigma = \operatorname{Re}(s) > 1$, there exists a real-valued $C^{\infty}$ function $M_{\sigma,1} : \C \rightarrow \mathbb{R}$ satisfying $M_{\sigma,1} (w) \geq 0$ and $ \; \int_{\C} M_{\sigma,1}(w) \; |dw| = 1$, such that
	\begin{equation}
		\text{Avg}_\chi \; \Phi(\mathcal{L}'(s,\chi)) \; = \; \int_\C M_{\sigma,1}(w) \; \Phi(w) \; |dw|,
	\end{equation}
	holds for any continuous function $\Phi$ on $\C$. 
\end{tm}
In Section \ref{Disthigh} we then discuss higher derivatives and key obstructions to generalizing this method to higher derivatives. We show that, for $m \geq 2$, distribution functions may still exist, but only for much larger values of $\sigma$, depending on $m$. We explicitly compute that, for $m = 2$, the distribution function $M_{\sigma, 2}$ for $\mathcal{L}^{(2)}(s,\chi)$ exists for $\sigma > 2.93$. 

\section{Distribution Functions: Some Background}
In this section, we present some background on distribution functions. The results presented are based on the paper \cite{jeswin}.

Let $\R^k$ be a $k$-dimensional Euclidean space, and let $\mathbf{x} = (x_1, \cdots, x_k)$ be a variable point.
\begin{df}
	A completely additive, non-negative set function $\phi(E)$ defined for all Borel sets $E$ in $\R^k$ and taking the value $1$ when $E = \R^k$ will be called a \emph{distribution function} in $\R^k$. 
\end{df}
\textbf{Notation.} An integral with respect to $\phi$ will be denoted by $$\displaystyle \int_E f(\textbf{x}) \phi(d \mathbf{x}),$$ and is to be understood in the Lebesgue-Radon (or Lebesgue-Stieltjes)
sense.
\begin{df}
	A set $E$ is called a \emph{continuity set} of $\phi$ if $\phi(E^\circ) = \phi( \overline{E})$, where $E^\circ$ denotes the set of all interior points of $E$ and $\overline{E}$ is the closure of $E$.
\end{df}
\begin{df}
	A sequence of distribution functions $\phi_n$ is said to be \emph{convergent} if there exists a distribution function $\phi$ such that $\phi_n(E) \rightarrow \phi(E)$ for all continuity sets $E$ of the limit function $\phi$, which is then unique. We will use the notation $\phi_n \rightarrow \phi$.
\end{df}
\begin{pr}
	A sequence of distribution functions $\{ \phi_n \}$ converges to a distribution function $\phi$ if and only if, for all bounded continuous functions $f$,
	$$\displaystyle \int_{\R^k} f(\textbf{x}) \phi_n(d \mathbf{x}) \rightarrow \displaystyle \int_{\R^k} f(\textbf{x}) \phi(d \mathbf{x}).$$
	Moreover, if $\phi_n \rightarrow \phi$ then, for every non-negative, continuous function $f$,
	$$\displaystyle \int_{\R^k} f(\textbf{x}) \phi (d \mathbf{x}) \leq \; \liminf  \displaystyle \int_{\R^k} f(\textbf{x}) \phi_n(d \mathbf{x}).$$
\end{pr}
\begin{df}
	If $\phi_1$ and $\phi_2$ are two distribution functions, then we define a new distribution function as their convolution, as follows: 
	$$\phi_1 * \phi_2 (E) \; := \;  \int_{\R^k} \phi_1(E - \mathbf{x}) \phi_2(d \mathbf{x}), $$
	for every Borel set $E$. Here $E-\mathbf{x}$ denotes the set obtained from $E$ by the translation $-\mathbf{x}$. 
\end{df}
One can show that $\phi_1 * \phi_2 = \phi_2 * \phi_1$.
\begin{df}
	The spectrum $S = S(\phi)$ of a distribution function $\phi$ is the set of points $\mathbf{x} \in \R^k$ for which $\phi(E)>0$ for any set $E$ containing $\mathbf{x}$ as an interior point. We note that $S$ is always a non-empty closed set.
\end{df}
\begin{df}
	The point spectrum $P = P(\phi)$ is the set of points $x$ such that $ \phi(\{x\}) >0$. 
\end{df}
\begin{df}
	A distribution function is called continuous if $P(\phi)$ is empty, and is called absolutely continuous if $\phi(E) = 0$ for all Borel sets $E$ of measure $0$. 
\end{df}
The following criterion is from \cite{jeswin}.
\begin{pr} A distribution function $\phi$ is absolutely continuous if and only if there exists a Lebesgue integrable point function $D(x)$ in $\R$ such that 
	$$ \phi(E) = \int_E D(x) dx,$$
	for any Borel set $E$. We call $D(x)$ the density function of $\phi$.
\end{pr}

\section{Construction of $M_{\sigma, P}$ functions}\label{DistConst}

Let $P$ be any finite set of non-archimedean primes of $K$ and set $T_P := \prod_{\wp \in P} \C^1$, where $\C^1$ denotes $\{ z : |z| = 1\}$. The following lemma is due to Ihara \cite[Lemma 4.3.1]{iharaM}. For the sake of completion, we include a proof here.
\begin{lm} \label{uniformchi}
	Let $K$ be as above and $\chi$ run over a family as above, excluding characters for which $\mathbf{f}_\chi \in P$. For each such $\chi$, let $\chi_P = (\chi(\wp))_{\wp \in P} \in T_P$. Then $(\chi_P)_\chi$ is uniformly distributed on $T_P$. That is, for any continuous function $\Psi : T_P \rightarrow \C$ we have, 
	$$ \text{Avg}_\chi \; \Psi(\chi_P) = \int_{T_P} \Psi(t_P) \; d^*t_P,$$
	where $d^* t_{\wp} = (2 \pi i t_{\wp})^{-1}d t_{\wp}$ is the normalized Haar measure on the $t_{\wp}$-unit circle.
	\vspace{2mm}
\end{lm}
\begin{proof}
	Let $Z_P = \prod_{\wp \in P}\Z$. For $n=(n_\wp)_{\wp \in P}\in Z_P$ and $t_P=(t_\wp)_{\wp \in P}\in T_P$, write
	$$t_P^n=\prod_{\wp \in P}t_\wp^{n_\wp}.$$
	By Weyl's criterion for the compact torus $T_P$, it is enough to prove that
	\begin{equation}\label{weyl-uniformchi}
		\text{Avg}_\chi \; \chi_P^n=0,
	\end{equation}
	for every non-zero $n\in Z_P$, where $\chi_P^n=\prod_{\wp \in P}\chi(\wp)^{n_\wp}$. Fix such an $n$ and put
	$$D_n=\prod_{\wp \in P}\wp^{n_\wp}.$$
	Then $D_n$ is a non-trivial fractional divisor supported on $P$, and $\chi_P^n=\chi(D_n)$ for every $\chi$ whose conductor is prime to $P$.
	
	Let $\mathbf{f}$ be a non-archimedean prime divisor not belonging to $P$. Averaging first over all characters whose conductor divides $\mathbf{f}$, the orthogonality relation gives
	\[
	T_{\mathbf{f}}:=
	\frac{\sum_{\mathbf{f}_\chi \mid \mathbf{f}}\chi(D_n)}
	{\sum_{\mathbf{f}_\chi \mid \mathbf{f}}1}
	=
	\begin{cases}
		1, & \text{if }D_n\equiv 1 \pmod{\mathbf{f}},\\
		0, & \text{otherwise}
	\end{cases}.
	\]
	Here $D_n\equiv 1 \pmod{\mathbf{f}}$ means that $D_n$ lies in the common kernel of all characters whose conductors divide $\mathbf{f}$. Since $D_n\neq (1)$ and the unit group of $K$ is finite, this congruence can hold for only finitely many prime divisors $\mathbf{f}$. Indeed, otherwise a fixed generator of $D_n$, up to multiplication by one of finitely many units, would be congruent to $1$ modulo infinitely many primes. This would force $D_n=(1)$.
	
	Now average over characters of conductor equal to $\mathbf{f}$:
	\[
	T'_{\mathbf{f}}:=
	\frac{\sum_{\mathbf{f}_\chi=\mathbf{f}}\chi(D_n)}
	{\sum_{\mathbf{f}_\chi=\mathbf{f}}1}.
	\]
	For prime $\mathbf{f}$, passing from conductor dividing $\mathbf{f}$ to conductor equal to $\mathbf{f}$ removes only the characters of smaller conductor, whose number is bounded in terms of $K$, while the number of characters modulo $\mathbf{f}$ is $\gg N\mathbf{f}$. Hence
	\[
	|T'_{\mathbf{f}}-T_{\mathbf{f}}|\ll (N\mathbf{f})^{-1},
	\]
	with an implied constant depending only on $K$.
	
	Using the definition of $\text{Avg}_\chi$ and ignoring the finitely many primes in $P$, which does not affect the limit, we obtain
	\[
	\text{Avg}_{N(\mathbf{f})\leq m}\chi(D_n)
	=
	\frac{\sum_{N(\mathbf{f})\leq m}T'_{\mathbf{f}}}
	{\sum_{N(\mathbf{f})\leq m}1}.
	\]
	The contribution of the $T_{\mathbf{f}}$ is supported on a finite set of prime divisors, and
	\[
	\frac{\sum_{N(\mathbf{f})\leq m}(N\mathbf{f})^{-1}}
	{\sum_{N(\mathbf{f})\leq m}1}\longrightarrow 0,
	\]
	by the prime ideal theorem. Therefore the preceding average tends to $0$ as $m\rightarrow \infty$, proving \eqref{weyl-uniformchi}. Weyl's criterion now gives the asserted uniform distribution.
\end{proof}
\begin{rem}
	The above lemma is the key ingredient of our results. The idea is to make a suitable change of variables in the above lemma, so that from the Jacobian a density function can be extracted.
\end{rem}
For any character $\chi$ of $K$ that is unramified at all primes in $P$, let
\[
L_P(\chi, s) = \prod_{\wp \in P} \; (1 - \chi(\wp) \; N \wp^{-s})^{-1}.
\]
Then
\[
\mathcal{L}_P(\chi, s) := \dfrac{L_P'(\chi,s)}{L_P(\chi, s)} = \sum_{\wp \in P} \;  - \dfrac{\chi(\wp) N \wp^{-s} \log N\wp }{ (1 - \chi(\wp) \; N \wp^{-s})};
\]
hence
\[
\mathcal{L}_P'(\chi, s) = \dfrac{d}{ds} \dfrac{L_P'(\chi,s)}{L_P(\chi, s)} = \sum_{\wp \in P} \; \dfrac{\chi(\wp) N\wp^{-s} (\log N\wp)^2}{(1 - \chi(\wp) N \wp^{-s})^2}.
\]
Looking at this, we define $g_{\sigma, P} : T_P \rightarrow \C$ by \begin{equation}\label{gsigmap}
	g_{\sigma, P} (t_P)= \sum_{\wp \in P} g_{\sigma, \wp}(t_{\wp}) \;\; \text{ where } \;\; g_{\sigma, \wp}(t_{\wp}) = \dfrac{ t_{\wp} N\wp^{\sigma} (\log N\wp)^2 }{(t_\wp - N\wp^\sigma)^2},
\end{equation}
where $t_P = (t_\wp)_{\wp \in P}$; in particular, $|t_\wp| = 1$. Thus we see that 
$$\mathcal{L}_P'(\chi, s) = g_{\sigma, P} (\chi_P NP^{-it}),$$
where $t = \operatorname{Im}(s)$ and $\chi_PNP^{-it} = (\chi(\wp) N\wp^{-it})_{\wp \in P}$.

For $(\mathbf{f}_\chi , P) = 1$, since $\{ \chi_P \}_\chi$ is uniformly distributed on $T_P$, so is its translate $\{\chi_PNP^{-it} \}_\chi$. Thus for any continuous function $\Phi$ on $\C$, by Lemma \ref{uniformchi}, applied to $\Psi = \Phi \; \circ \; g_{\sigma, P}$, we obtain
\begin{equation} \label{LPeqn}
	\text{Avg}_\chi \left( \Phi \left( \mathcal{L}_P' (\chi, s) \right)\right) = \int_{T_P} \Phi(g_{\sigma, P}(t_P)) \; d^*t_P.
\end{equation}
\vspace{2mm}\par We first note the following.

\begin{lm} \label{LP}
	For fixed $s$, with $\sigma = \operatorname{Re}(s)>1$, and for $P = P_y = \{ \wp \; : \; N\wp \leq y\}$, as $y \rightarrow \infty$, $\mathcal{L}_P'(\chi, s)$ tends uniformly to $\mathcal{L}'(\chi, s).$
\end{lm}
\begin{proof}
	For any $\chi$ we have,
	$$| \mathcal{L}'(\chi, s) - \mathcal{L}_P'(\chi, s) | \; \leq \; \sum_{\wp \notin P} \dfrac{N\wp^{\sigma}(\log N\wp)^2}{(N\wp^\sigma - 1)^2}.$$
	Thus letting $y \rightarrow \infty$, the right-hand side tends to $0$.
\end{proof}
\vspace{2mm}
\begin{tm} \label{Msigmap}
	Let $\sigma > 0$. Then there exists a function $M_{\sigma, P} : \C \rightarrow \R$ such that, for any continuous function $\Phi(w)$ on $\C$,
	$$ \int_\C M_{\sigma, P}(w) \Phi(w) |dw| \; = \; \int_{T_P} \Phi(g_{\sigma, P}(t_P)) \; d^*t_P,$$
	where $w = x+iy$ and $|dw| = (2 \pi)^{-1} dx dy$, and $d^* t_P$ is the normalized Haar measure on $T_P$. The function $M_{\sigma, P}$ is compactly supported and satisfies the following properties:
	\begin{enumerate}
		\item $ M_{\sigma, P} (w) \geq 0,$ 
		\item $ M_{\sigma, P}(\overline{w}) = M_{\sigma, P}(w), $
		\item $ \displaystyle \int_{\C} M_{\sigma, P}(w) \; |dw| = 1.$
	\end{enumerate}
\end{tm}
\begin{proof}
	We first consider the case when $|P|=1$, say $P = \{ \wp \}$. Let $T_\wp = \C^1$ and write $t_\wp = e^{i \theta}$, so $d^* t_\wp = \frac{1}{2 \pi} d \theta$.
	
	In the open unit disc, write $z = r e^{i \theta}$ for $0 \leq r < 1$ and $0 \leq \theta < 2 \pi$, and consider the map $$w = w(z) = \dfrac{ (\log  N\wp)^2 \; r e^{i \theta} }{(1- re^{i \theta})^2} = \dfrac{ A \; r e^{i \theta} }{(1- re^{i \theta})^2}.$$ 
	
	For brevity, write $A = (\log  N\wp)^2$. Let $\rho$ be a real number such that $ N \wp^{- \sigma} < \rho < 1$ and let $B_\rho$ be the region bounded by the curve $$w =  \frac{ A \; \rho e^{i \theta} }{(1- \rho e^{i \theta})^2}.$$
	Thus $w = w(z)$ gives a one-to-one correspondence between the region $B_{\sigma, \wp}$ and the disc $C_{\rho} := \{z \; : \; |z| < \rho \}$.
	
	Let us now compute the Jacobian of this mapping. We see that, 
	\begin{equation}\label{w(z)}
		w(z) = A \; \dfrac{r \cos \theta - 2 r^2 + r^3 \cos \theta}{ |1 - re^{i \theta}|^4} + i \; A \; \dfrac{r \sin \theta - r^3 \sin \theta}{| 1 - r e^{i \theta }|^4} = U + iV \;\;\;\; \text{(say)}.
	\end{equation} 
	Thus the Jacobian is given by $$ J \; = \; \begin{vmatrix}
		\frac{\partial U}{\partial r} & \frac{\partial U}{\partial \theta} \\[0.5em]
		\frac{\partial V}{\partial r} & \frac{\partial V}{\partial \theta} 
	\end{vmatrix} \;  = \; \dfrac{A^2 \; r \; |1+re^{i\theta}|^2}{|1-re^{i \theta} |^6}.
	$$
	{ \emph{This Jacobian was computed using a computer algebra system.} \vspace{2mm} }	
	Thus we have 
	\begin{align*}
		\int_{T_{\wp}} \Phi(g_{\sigma, \wp} (t_{\wp})) \; d^*t_{\wp} \; &= \dfrac{1}{2 \pi} \int_{0}^{2 \pi } \Phi \left( \dfrac{ e^{i \theta} N\wp^{\sigma} (\log N\wp)^2 }{( e^{i \theta} - N\wp^\sigma)^2} \right) \; d \theta \\[1.3ex]
		& = \dfrac{1}{2 \pi} \int_{0}^{2 \pi } \Phi \left( \dfrac{ e^{i \theta} N\wp^{ - \sigma} (\log N\wp)^2 }{( 1 -  N\wp^{ -\sigma}  e^{i \theta} )^2} \right) \; d \theta \\[1.3ex]
		& = \dfrac{1}{2 \pi} \int \int_{B_{\sigma, \wp}} \Phi(w) \delta(r - N \wp^{- \sigma}) \; J^{-1} \; dU \; dV,
	\end{align*}
	where $\delta(\cdot)$ denotes the Dirac delta distribution and $w = U + iV$. Therefore we define $M_{\sigma, \wp}(w)$ in the following way: 
	\begin{equation}
		M_{\sigma, \wp}(w) = J^{-1} \delta(r - N \wp^{- \sigma}) = \dfrac{|1 - r e^{i \theta}|^6}{ (\log N\wp)^2 \; |1 + r e^{i \theta}|^2} \; \dfrac{\delta(r - N\wp^{- \sigma})}{r},
	\end{equation}
	for $w \in B_{\sigma, \wp}$ and $M_{\sigma, \wp}(w) = 0$ otherwise. Substituting this expression, we obtain
	$$  \int_{T_\wp} \Phi(g_{\sigma, \wp}(t_\wp)) \; d^*t_\wp \; = \; \int_\C M_{\sigma, \wp}(\omega) \; \Phi(\omega) \; |d\omega|.$$
	This proves the case $P = \{ \wp \}$. For the general case, we define the function using the convolution product. That is,
	$$M_{\sigma, P}(w) = \ast_{\wp \in P} \; M_{\sigma, \wp}(w).$$
	In other words, for $P = P' \cup \{ \wp\}$ define 
	\begin{equation}
		M_{\sigma, P}(w) = \int_{\C} M_{\sigma, P'}(w') \; M_{\sigma, \wp}(w - w') \; |dw'|.
	\end{equation}
	Note that, for any open set $U \subseteq \C$ we obtain, 
	\begin{equation}
		\int_{U} M_{\sigma, P}(w) \; |dw| = \; \text{Vol}(g_{\sigma, P}^{-1} (U)),
	\end{equation}
	where the volume is with respect to $d^*t_P$ and thus $\int_{\C} M_{\sigma, P}(w) \; |dw| \; = 1$. The Haar measure is normalized, i.e., the total volume of $T_P$ is $1$.
\end{proof}
\vspace{3mm}
Our next goal is to show that, for $P = P_y = \{ \wp \; : \; N\wp \leq y\}$ as before, $M_{\sigma, P_y}(w)$ converges uniformly in $w$ to a function $M_{\sigma}(w)$ as $y \rightarrow \infty$. \vspace{2mm}

\begin{pr}\label{Msigmauni}
	For $P=P_y$, as $y \rightarrow \infty$ and for $\sigma > 1/2$, ${M}_{\sigma, P}(w)$ converges to $M_{\sigma}(w)$ uniformly in $w$. The limit $M_{\sigma}(w)$ is therefore continuous in $w$ and non-negative.
\end{pr}
\begin{proof}
	For $\wp \notin P$, writing $N\wp^{-\sigma} = q$, we have
	\begin{align*}
		|M_{\sigma, P \cup \{ \wp\}} (w) - M_{\sigma, P}(w)| & = \left| \frac{1}{2 \pi} \frac{1}{(\log N\wp)^2}\int_0^{2 \pi} \frac{|1-q e^{i \theta}|^6 }{|1+q e^{i \theta} |^2} M_{\sigma, P}(z - q e^{i \theta}) \; d\theta \right| \\[1.3ex]
		& \ll \frac{q^4}{(\log N \wp)^2} \ll \left( \dfrac{1}{ N\wp^{\sigma}}\right)^4.
	\end{align*}
	Note that, by (1) and (3) of Theorem \ref{Msigmap}, $M_{\sigma, P}$ is bounded. Thus, ${M}_{\sigma, P}(w)$ converges uniformly to a function, say $M_{\sigma}(w)$, for $\sigma > 1/2$ (in fact, for $\sigma > 1/4$). 
\end{proof}
\begin{rem}
	We also have $\int_{\C}M_{\sigma}(w)|dw| = 1$. However, we will show this after proving the next theorem. Note that, since $\int_{\C}M_{\sigma, P}(w)|dw| = 1$ for all $P$, the uniform convergence already gives, $$\int_{\C} M_{\sigma} (w)|dw| \leq 1.$$
\end{rem}
\begin{tm}\label{dist1}
	For any $s \in \C$ with $\sigma = \operatorname{Re}(s) > 1$ 
	\begin{equation}
		\text{Avg}_\chi \; \Phi(\mathcal{L}'(\chi, s)) \; = \; \int_\C M_{\sigma}(w) \; \Phi(w) \; |dw|,
	\end{equation}
	holds for any continuous function $\Phi$ of $\C$.
\end{tm}
\begin{proof}
	From \eqref{LPeqn} and Theorem \ref{Msigmap} we have, 
	\begin{align*}
		\text{Avg}_\chi \left( \Phi \left( \mathcal{L}_P' (\chi, s) \right)\right) &= \int_{T_P} \Phi(g_{\sigma, P}(t_P)) \; d^*t_P \\[1.3ex]
		&= \int_{\C} M_{\sigma, P}(w) \Phi(w) |dw|.
	\end{align*}
	The theorem follows by taking the limit, using Lemma \ref{LP} and Proposition \ref{Msigmauni}. 
\end{proof} 
\vspace{3mm} Note that if we take the particular case of $\Phi(w) = P^{(a,b)}(w) = w^a \overline{w}^b$, then the moment results give (see \cite[Theorem 1.2]{ghosh2})
$$ (-1)^{(1+1)(a+b)} {\mu}^{(a,b)}(1) = \int_{\C} M_{\sigma}(w) P^{(a,b)}(w) |dw|.$$
In particular, taking $a=b=0$ gives,
$$\int_{\C}M_{\sigma}(w) |dw| = {\mu}^{(0,0)}(1) = 1.$$

We also note that if we consider the Fourier dual of $M_{\sigma}(z)$ given by 
$$\tilde{M}_{\sigma}(z) = \int_{\C} M_{\sigma}(w) \psi_z(w)|dw|.$$
By Theorem \ref{dist1}, we have
$$\tilde{M}_{\sigma}(z) = \text{Avg}_\chi \; \psi_z(\mathcal{L}'(\chi, s)).$$

\section{Higher derivatives} \label{Disthigh}

The above technique can in principle be generalized to higher derivatives. It remains to choose the $g_{\sigma, \wp}(t_\wp)$ function appropriately so that, for a local factor, we obtain 
$$\mathcal{L}_\wp^{(n)}(\chi, s) = g_{\sigma, \wp}(\chi(\wp) N \wp^{-it}).$$ 
However, for higher derivatives, computing these $M_{\sigma, P}$ functions explicitly becomes very involved. Even in our case, we used a computer algebra system to simplify the Jacobian. 
Writing $u = \chi(\wp)N\wp^{-s}$ and hence $\frac{du}{ds} = -u \log N\wp$, we obtain
\[
\mathcal{L}_{\wp}(s, \chi) = \dfrac{- \chi(\wp) N\wp^{-s} \log N\wp}{(1-\chi(\wp)N\wp^{-s})} = \dfrac{ - (\log N\wp ) u}{1-u}.
\]
Define
\[
g(u) := \frac{u}{1-u}, \quad \text{so that } \mathcal{L}_{\wp}(s, \chi) = (-\log N\wp) \, g(u), \quad u' = -u \log N\wp, \quad u^{(r)} = u (- \log N\wp)^r.
\]
Faà di Bruno's formula gives
$$
\frac{d^r}{ds^r} g(u) = \sum_{k=1}^r g^{(k)}(u) \, B_{r,k}\big(u', u'', \dots, u^{(r-k+1)}\big),
$$
where $B_{r,k}$ denote the partial Bell polynomials. Here,
$$
g^{(k)}(u) = \frac{k!}{(1-u)^{k+1}}, \qquad \text{for } k \geq 1.
$$
Thus we have 
\begin{align*}
	\mathcal{L}^{(r)}_{\wp}(s, \chi) &= (-\log N\wp) \sum_{k=1}^r \frac{k!}{(1-u)^{k+1}} \, B_{r,k}\big(u', u'', \dots, u^{(r-k+1)}\big) \\
	&= (-\log N\wp) \sum_{k=1}^r \frac{k!}{(1-u)^{k+1}} \, B_{r,k}\big((-\log N\wp) u, (-\log N\wp)^2 u, \dots, (-\log N\wp)^{r-k+1} u \big) \\
	&= (-\log N\wp)^{r+1} \sum_{k=1}^r \frac{k!}{(1-u)^{k+1}} \, B_{r,k}(u, u, \dots, u) \\
	&= (-\log N\wp)^{r+1} \sum_{k=1}^r \frac{(k!)u^k}{(1-u)^{k+1}} \, B_{r,k}(1, 1, \dots, 1)\\
	&= (-\log N\wp)^{r+1} \sum_{k=1}^r \frac{u^k}{(1-u)^{k+1}} (k!)^2 S(r,k) \\ 
	&= (-\log N\wp)^{r+1} \sum_{k=1}^r \frac{(\chi(\wp)N\wp^{-s})^k}{(1-\chi(\wp)N\wp^{-s})^{k+1}} (k!)^2 S(r,k).
\end{align*}

Thus, as before, we can define $g_{\sigma, r, P} : T_P \rightarrow \C$ by $g_{\sigma, r, P} (t_P)= \sum_{\wp \in P} g_{\sigma, r, \wp}(t_{\wp})$ where 
\begin{equation}\label{gsimarp}
	g_{\sigma, m, \wp}(t_{\wp}) = (-\log N\wp)^{m+1} \sum_{k=1}^m \frac{ N\wp^{-\sigma k}t_{\wp}^k}{(1 - N\wp^{-\sigma}t_{\wp})^{k+1}} (k!)^2 S(m,k).
\end{equation}
We note that for $m=1$, $g_{\sigma, 1, \wp}(t_{\wp}) = g_{\sigma, \wp}(t_{\wp})$ agrees with \eqref{gsigmap}.

We further note that 
$$\mathcal{L}_P^{(r)}(\chi, s) = \sum_{\wp \in P} (-\log N\wp)^{r+1} \sum_{k=1}^r \frac{(\chi(\wp)N\wp^{-s})^k}{(1-\chi(\wp)N\wp^{-s})^{k+1}} (k!)^2 S(r,k) = g_{\sigma, r, P}(\chi_P NP^{-it}),$$
where, as before, $t = \operatorname{Im}(s)$ and $\chi_PNP^{-it} = (\chi(\wp) N\wp^{-it})_{\wp \in P}$. 
\begin{lm} \label{LP-high}
	For fixed $s$, with $\sigma = \operatorname{Re}(s)>1$, and for $P = P_y = \{ \wp \; : \; N\wp \leq y\}$, as $y \rightarrow \infty$, $\mathcal{L}_P^{(r)}(\chi, s)$ tends uniformly to $\mathcal{L}^{(r)}(\chi, s).$
\end{lm}
\begin{proof}
	For any fixed $r$ and $\chi$ we have,
	$$| \mathcal{L}^{(r)}(\chi, s) - \mathcal{L}_P^{(r)}(\chi, s) | \; \leq \; \sum_{\wp \notin P} (\log N\wp)^{r+1} \sum_{k=1}^r \frac{N\wp^{\sigma}}{(N\wp^{\sigma}-1)^{k+1}} (k!)^2 S(r,k).$$
	Thus, letting $y \rightarrow \infty$, the right-hand side tends to $0$, as $\sigma > 1$.
\end{proof}
\subsection{Obstruction to generalization to higher derivatives}
As before, we consider the case when $|P|=1$, say $P = \{ \wp \}$. Let $T_\wp = \C^1$ and write $t_\wp = e^{i \theta}$, so $d^* t_\wp = \frac{1}{2 \pi} d \theta$. To deduce the density function, ideally we would make a change of variables and be able to write
\begin{equation} \label{obs1}
	\int_{T_{\wp}} \Phi(g_{\sigma,m, \wp} (t_{\wp})) = \dfrac{1}{2 \pi} \int \int_{B_{\sigma,m, \wp}} \Phi(w_m) \delta(r - N \wp^{- \sigma}) \; J^{-1} \; dU \; dV ,
\end{equation}
where, in the open unit disc, $z = r e^{i \theta}$ for $0 \leq r < 1$ and $0 \leq \theta < 2 \pi$, we define $$w_m = w_m(z) = (-\log N\wp)^{m+1} \sum_{k=1}^m \frac{ (r e^{i\theta})^k}{(1-re^{i \theta})^{k+1}} (k!)^2 S(m,k).$$ 
If we write $w_m(z) = U + iV$, then the $J$ in \eqref{obs1} is the Jacobian given by $$ J \; = \; \begin{vmatrix}
	\frac{\partial U}{\partial r} & \frac{\partial U}{\partial \theta} \\[0.5em]
	\frac{\partial V}{\partial r} & \frac{\partial V}{\partial \theta} 
\end{vmatrix}.
$$
However, this change-of-variables argument fails, because for $m>1$, $w_m(z)$ is no longer a one-to-one map between the region $B_{\sigma, \wp}$ and the disc $C_{\rho} := \{z \; : \; |z| < \rho \}$. For larger values of $\sigma$, depending on $m$, it is still possible to derive density functions. For example, let us consider the case of $m=2$ in detail. Noting that $S(2,1) = 1 = S(2,2)$, we have
\begin{align*}
	w_2(z) &= -(\log N\wp)^3 \left( \dfrac{re^{i\theta}}{(1-r e^{i \theta})^2} + \dfrac{4r^2 e^{2i\theta}}{(1-r e^{i \theta})^3} \right)\\
	&= -(\log N\wp)^3 \left( \dfrac{z}{(1-z)^2} + \dfrac{4z^2}{(1-z)^3} \right)\\
	&= -(\log N\wp)^3 \dfrac{(z+3z^2)}{(1-z)^3}, \text{ hence}\\
	w_2'(z) &= \frac{1+8z+3z^2}{(1-z)^4}.
\end{align*}
The numerator of $w_2'(z)$ vanishes when $1+8z+3z^2=0$, for $z = (-4 \pm \sqrt{13})/3$. A zero of $w_2'(z)$ inside the disc implies local non-injectivity. To avoid this, we need to be able to choose $\rho$ such that $\rho < \big| (-4 + \sqrt{13})/3 \big| \approx 0.1315$. Since $N \wp^{-\sigma} < \rho$, we should have $\sigma > - \log_2 (0.1315) \approx 2.93$. We note that this is the general situation: the numerator of $w_m'(z)$ is a polynomial with a large leading coefficient. Consequently, it will always have a zero very close to the origin. Thus for higher derivatives, we can only obtain a density function for large values of $\sigma$. For large $m$, we expect that $\sigma$ needs to be polynomially large in $m$. Hence, this method fails to yield distribution results for $\sigma > 1$. 
\iffalse
We give a rough estimate of how large $\sigma$ needs to be, depending on $m$. We note that $$\dfrac{z^k}{(1-z)^{k+1}} = \sum_{j \geq k} \binom{j}{k}z^j $$ Thus $w_m(z)$ has the following power series centered at 0. 
$$w_m(z) = \sum_{l \geq 1} a_l z^l, \;\; \text{ where } \;\; a_l = \sum_{k=1}^{\min{(m,l)}} (k!)^2 S(m,k) \binom{l}{k}$$
In particular we note that $w_m(0) = 0$ and $w_m'(0) = a_1 = (1!)^2 S(m,1) \binom{1}{1} = 1$. Thus 
$$w_m'(z) = 1 + \sum_{l \geq 2} l a_l z^{l-1}$$
For $|z|< \rho < 1$, $|w_m'(z) - 1| \leq \sum_{l \geq 2} l |a_l| \rho^{l-1} $. Therefore, if $\sum_{l \geq 2} l |a_l| \rho^{l-1}<1$, then $|w_m'(z)|>0$, giving us injectivity within the disc of radius $\rho$. Now we bound the coefficients.
Since $S(m,k) \leq \frac{k^m}{k!}$, we have $|a_l| \leq \sum_{k \leq \min (m,l)} k^m l^k  \leq m^{m+1}l^m $. Thus, $$\sum_{l \geq 2} l |a_l| \rho^{l-1} \leq m^{m+1} \sum_{l \geq 2} l^{m+1} \rho^{l-1} \leq m^{m+1} \frac{m!}{(1-\rho)^{m+1}}.$$
Thus, 
\\
\fi
\subsection{Remarks on extending the results to $\frac{1}{2} < \sigma \leq 1$}
We also note that for $\sigma >1$, the image of $g_{\sigma, P}$ remains bounded as $|P| \rightarrow \infty$. Since the support of $M_{\sigma, P}$ is the image of the mapping, $g_{\sigma, P}$, the support of $M_{\sigma}$ is also bounded. Therefore, in the proof of the above theorem we may assume $\Phi$ to be continuous.

This is no longer true for $\sigma > \frac{1}{2}$; that is, the image of $g_{\sigma, P}$ need not be bounded. As remarked earlier, in a later paper \cite{iharamatsu}, Ihara and Matsumoto extended their result to $\sigma > 1/2$ under GRH and some conditions on the test function. They introduced the idea of admissible functions and developed a more general notion of $g_{\sigma, \wp}$ which they called $g$-functions. However, this approach does not seem to generalize for higher derivatives. It fails at essential steps in Sections 3.1 and 3.3 of their paper. We have yet to discover a way of doing this, and this remains work in progress.
\section*{Acknowledgements}
The author sincerely thanks Prof. V. Kumar Murty for his valuable suggestions, encouragement and many insightful discussions during the developement of this work. The author is currently partially supported by NSERC grants RGPIN-2018-03770, RGPIN-2020-06075, and by CRC tier-2 research stipend 950-231716 at the University of Calgary. The author thanks Prof. Khoa D. Nguyen for their financial support and mentorship.

\bibliographystyle{amsplain}
\bibliography{EulerKron}
\end{document}